\documentclass[11pt, twoside]{amsart}

\usepackage{amsmath}
\usepackage{amsfonts}
\usepackage{amssymb}
\usepackage{amsthm}
\usepackage{bbm}
\usepackage{verbatim} 
\usepackage{graphicx}
\usepackage{caption}
\usepackage{placeins}

\usepackage[T1]{fontenc}

\theoremstyle{plain}
\newtheorem{thm}{Theorem}[section]
\newtheorem{cor}[thm]{Corollary}
\newtheorem{lma}[thm]{Lemma}

\newtheorem{fact}[thm]{Fact}

\theoremstyle{definition}

\theoremstyle{remark}

\numberwithin{equation}{section}

\newcommand{\mcA}{\mathcal{A}}

\newcommand{\mcG}{\mathcal{G}}
\newcommand{\mcH}{\mathcal{H}}

\newcommand{\mcM}{\mathcal{M}}

\newcommand{\mbbN}{\mathbb{N}}

\newcommand{\mbbZ}{\mathbb{Z}}

\title{$> k-$homogeneous infinite graphs}
\author{Ove Ahlman}
\address{Ove Ahlman, Department of Mathematics, Uppsala University, Box 480, 75 106 Uppsala, Sweden}
\email{ove@math.uu.se}
\keywords{>k-homogeneous, classification, countably infinite graph}
\subjclass[2010]{05C75, 05C63, 03C50}

\begin{document}
\maketitle
\begin{abstract}
In this article we give an explicit classification for the countably infinite graphs $\mcG$ which are, for some $k$, $\geq$$ k$-homogeneous. It turns out that a $\geq$$k-$homogeneous graph $\mcM$ is non-homogeneous if and only if it is either not $1-$homogeneous or not $2-$homogeneous, both cases which may be classified using ramsey theory.
\end{abstract}
\section{introduction}
A graph $\mcG$ is called \textbf{$k-$homogeneous} if for each induced subgraph $\mcA\subseteq \mcG$ such that $|A|=k$ and embedding $f:\mcA\rightarrow \mcG$, $f$ may be extended into an automorphism of $\mcG$. If $\mcG$ is $t-$homogeneous for each $t\geq k$ ($t\leq k$) then $\mcG$ is called \textbf{$\geq$$k-$homogeneous} (\textbf{$\leq$$k-$homogeneous}). A graph which is both $\geq$$k-$homogeneous and $\leq$$k-$homogeneous is plainly called homogeneous. Lachlan and Woodrow \cite{LV} classified the countably infinite homogeneous graphs. Since then, the study of homogeneous structures has been continued in many different ways. When it comes to countably infinite homogeneous structures Lachlan \cite{L} classified all such tournaments and Cherlin \cite{Ch} classified all such digraphs. Even more kinds of infinite homogeneous structures have been classified, however few results about infinite $k-$homogeneous structures which are not homogeneous seem to exist. When it comes to finite structures Gardiner \cite{G} and independently Golfand and Klin \cite{GK} classified all homogeneous finite graphs. Cameron \cite{Ca} extended this to the $k-$homogeneous context and showed that any $\leq$$5-$homogeneous graph is homogeneous. Thus classifying the finite $k-$homogeneous graphs comes down to the cases which are not $t-$homogeneous for some $t\leq 5$. Among others, Chia and Kok \cite{CK} took on this task and characterized finite $k-$homogeneous graphs with a given number of isolated vertices and nontrivial components. In general though no known characterization of the finite $k-$homogeneous, $\geq$$k-$homogeneous or $\leq$$k-$homogeneous graphs exist. In the present article however we do make progress in the subject when it comes to infinite graphs, and provide a full classification of all $\geq$$k-$homogeneous infinite graphs.
\\\indent
For each $t\in\mbbZ^+$ define the graph $\mcG_t$ as having universe $G_t = \mbbZ \times \{1,...,t\}$ and edges $E = \{\{(a,i),(b,j)\}: a\neq b\}$. Notice that $\mcG_t$ may also be described as the complement of the graph which consists of $\omega$ disjoint copies of $K_t$.
If $t\geq 2$ let $\mcH_{t,1}$ be the graph with universe $H_{t,1} = \mbbZ\times \{1,...,2t\}$ and edges 
\[E_{t,1} = \{\{(a,i),(b,j)\} : i,j\leq t \text{ or } i,j > t, \text{ and } a\neq b \}. \]
Let $\mcH_{t,2}$ have the same universe as $\mcH_{t,1}$ but with edge set
\[E_{t,2} = E_{t,1} \cup \{\{(a,i),(b,j)\}: i\leq t, j>t \text{ and } a=b\}.\]
Lastly define the graph $\mcH_{1,2}$ as having universe $\mbbZ\times \{1,2\}$ and edge set $E= \{\{(a,i),(b,j)\}: i = j \text{ or } a=b\}$.
\begin{figure}
\begin{minipage}{0.45\textwidth}
\centering
\includegraphics{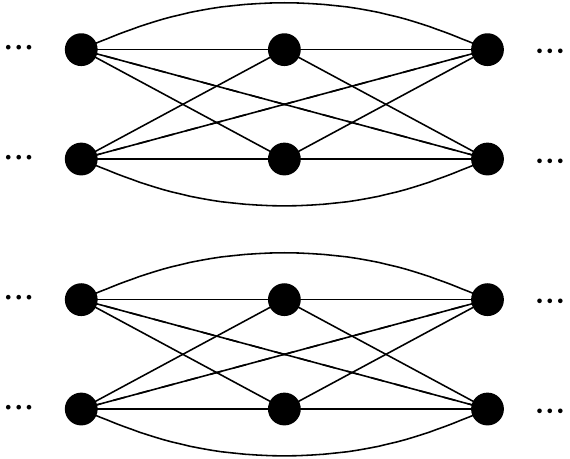}
\caption{The graph $\mcH_{2,1}$}
\end{minipage}\hfill
\begin{minipage}{0.45\textwidth}
\centering
\includegraphics{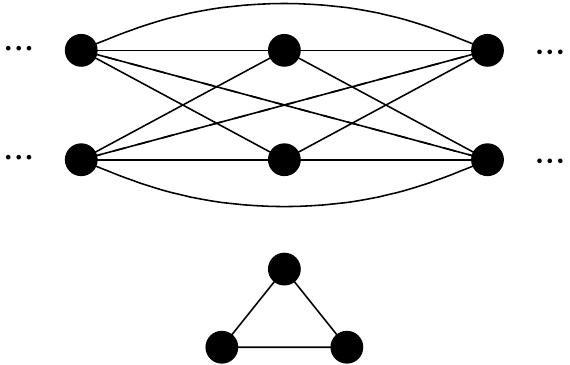}
\caption{The graph $G_2 \dot\cup K_3$}
\end{minipage}
\end{figure}
\begin{thm}\label{thm} Let $\mcM$ be a countably infinite graph. Then $\mcM$ is $\geq$$k-$homogeneous, for some integer $k\geq 1$, if and only if exactly one of the following hold.
\begin{itemize}
\item[(i)] $\mcM$ is a homogeneous graph.
\item[(ii)] $\mcM$ is not $1-$homogeneous and for some finite homogeneous graph $\mcH$ and some $t$ we have that $\mcM\cong \mcG_t\dot\cup\mcH$ or $\mcM\cong (\mcG_t\dot\cup\mcH)^c$.
\item[(iii)] $\mcM$ is $1-$homogeneous but not $2-$homogeneous and for some $t$ we have that $\mcM\cong\mcH_{t,1}$, $\mcM\cong\mcH_{t,2}$, $\mcM\cong\mcH_{t,1}^c$ or $\mcM\cong\mcH_{t,2}^c$.
\end{itemize} 
\end{thm}

\noindent As the finite homogeneous graphs are classified, in \cite{G} and \cite{GK}, as either the $3\times 3$ rook graph\footnote{The line graph of the complete bipartite graph with 3 vertices in each part i.e. how a rook moves on a 3 by 3 chessboard}, the $5-$cycle, a disjoint union of complete graphs or the complement of one of the previous graphs, case (ii) is complete.
We prove (ii) in Section \ref{sec1} Lemma \ref{thm1}, (iii) in Section \ref{sec2} Lemma \ref{twothm} and lastly in Section \ref{sec3} Lemma \ref{thirdmainlemma} we show that any $\geq$$k-$homogeneous graph which does not fit into (ii) or (iii) has to be homogeneous, in other words, (i) is proven. \\\indent 
A graph $\mcG$ is \textbf{homogenizable} if there exists a homogeneous structure $\mcM$ with a finite amount of extra relational symbols in its signature compared to $\mcG$ such that the automorphism groups of $\mcM$ and $\mcG$ are the same and if we remove all extra relations from $\mcM$, we get $\mcG$. For a more detailed definition of homogenizable structures and explicit examples see \cite{A,M}. From the proof of Theorem \ref{thm} we can draw the following two corollaries which relate to being homogenizable.

\begin{cor}\label{onecor}
If, for some $k$, $\mcM$ is a countably infinite $\geq$$k-$homogeneous graph which is not $1-$homogeneous then $\mcM$ is homogenizable by only adding a single unary relation symbol.
\end{cor}
\begin{cor}\label{twocor}
If, for some $k$, $\mcM$ is a countably infinite $\geq$$k-$homogeneous graph which is $1-$homogeneous but not $2-$homogeneous then $\mcM$ is homogenizable by adding only a single binary relation symbol and one of the following holds:
\begin{itemize}
\item $k= 5$ and $\mcM\cong \mcH_{1,2}$ or $\mcM\cong \mcH_{1,2}^c$ .
\item $k = 2n+1>3$ and $\mcM \cong \mcH_{n, 1}$ or $\mcM\cong \mcH_{n,1}^c$.
\item $k = 4n+1>5$ and $\mcM\cong\mcH_{n,2}$ or $\mcM\cong \mcH_{n,2}^c$.
\end{itemize}
\end{cor}
Note that an overlap between cases is to be expected. Corollary \ref{onecor} does not have an explicit classification of $\mcM$ depending on $k$, such as Corollary \ref{twocor}, since $\mcM\cong\mcG_t\dot\cup\mcH$ or $\mcM\cong(\mcG_t\dot\cup\mcH)^c$ where both $t$ and $\mcH$ affect for which $k$ that $\mcM$ is $\geq$$k-$homogeneous. It is clear in Corollary \ref{twocor} that we may add a unary relation symbol to make the graph homogeneous. This however does not follow the definition of being homogenizable since adding a unary relation symbol changes the automorphism group. In the terminology of Cherlin \cite{Ch2} we have proven that a $\geq$$k-$homogeneous graph has relational complexity at most $2$.

\subsection*{Notation and terminology}
%\section{Preliminaries}
For each $t\in\mbbZ^+$, $K_t$ is the complete graph on $t$ vertices and $K_\infty$ is the countably infinite complete graph. \emph{Whenever we talk about subgraphs $\mcH\subseteq \mcG$ we mean induced subgraph in the sense that for $a,b\in\mcH$ we have that $aE^\mcG b$ if and only if $aE^\mcH b$.} An embedding of graphs $f:\mcG\rightarrow \mcH$ is an injective function such that for each $a,b\in\mcG$ we have $aE^\mcG b$ if and only if $f(a)E^\mcH f(b)$. If we write $K_t\subseteq \mcG$ or $K_\infty\subseteq\mcG$ it means that for some subgraph $\mcH$ of $\mcG$, $\mcH$ is isomorphic to $K_t$ or $K_\infty$ respectively. If $\mcG$ is a graph with $a_1,...,a_r\in G$ then $K_t(a_1,...,a_r)$ is a complete subgraph of $\mcG$ containing $t$ vertices, which includes $a_1,...,a_r$. For $t\in \mbbZ^+$, a $t-$orbit of $\mcG$ is an orbit of $t-$tuples which arise when the automorphism group of $\mcG$ acts on $\mcG^t$. One of our main tools in the proofs is Ramsey's famous theorem about the existence of infinite complete or infinite independent subgraphs which now a days is common practice and may be found in for instance \cite{GRS}.
\begin{fact}[Ramsey's Theorem] If $\mcG$ is an infinite graph then $K_\infty\subseteq \mcG$ or $K_\infty^c\subseteq \mcG$.
\end{fact}

\section{Graphs which are not $1-$homogeneous}\label{sec1}
\begin{lma}\label{thm1} For some $k\in \mbbN$, $\mcM$ is a $\geq$$k-$homogeneous graph which is not $1-$homogeneous if and only if there exists $t\in \mbbN$ and a finite homogeneous graph $\mcH$ such that $\mcM \cong \mcG_t\dot\cup\mcH$ or $\mcM \cong (\mcG_t\dot\cup\mcH)^c$.
\end{lma}
The proof of this lemma is left for the end of this section. In the rest of this section we assume that $\mcM$ is $\geq$$k-$homogeneous but not $1-$homogeneous, thus $\mcM$ has more than one $1-$orbit. Due to Ramsey's theorem $K_\infty$ or $K_\infty^c$ is embeddable in $\mcM$. We will assume that $K_\infty$ is embeddable in $\mcM$. The reader may notice that all the reasoning in this section may be done in the same way if $K_\infty^c$ would be embeddable by switching all references of edges and non-edges, thus producing a result for the complement. 
\begin{lma}
There are exactly two $1-$orbits in $\mcM$ and one of them is finite.
\end{lma}
\begin{proof}
Since $\mcM$ is $\geq$$k-$homogeneous, all elements which are in $K_\infty\subseteq \mcM$ have to be in the same $1-$orbit, call it $p$. Assume that $a,b\notin p$ and note that $a$ and $b$ are adjacent to at most $k-1$ elements in $K_{3k}\subseteq\mcM$. Thus we may find $\mcG\subseteq K_{3k}\subseteq \mcM$ such that $|\mcG| = k$ and no element in $\mcG$ is adjacent to $a$ or $b$. Let $f:\mcG\cup \{a\}\rightarrow \mcG\cup \{b\}$ be the embedding which maps $\mcG$ to $\mcG$ and $a$ to $b$. Since $\mcM$ is $\geq$$k-$homogeneous $f$ may be extended to an automorphism, thus $a$ and $b$ belong to the same orbit.
\\\indent The orbit $p$ has to be infinite since all elements in $K_\infty\subseteq\mcM$ belong to $p$. Assume that the second orbit, call it $q$, is also infinite. By Ramsey's theorem either $K_\infty$ or $K_\infty^c$ is embeddable in $q$, however $K_\infty$ is impossible since $\mcM$ is $\geq$$k-$homogeneous and the orbits $p$ and $q$ are distinct. 
Each element in $K_\infty^c\subseteq q\subseteq \mcM$ is adjacent to at most $k-1$ elements in $K_k\subseteq p$. However then there has to be $a\in K_k\subseteq p$ and $\mcG\subseteq K_\infty^c\subseteq q$ such that $|\mcG|=k$ and $a$ is not adjacent to any element in $\mcG$. But then any embedding $f:\mcG\cup\{a\}\rightarrow \mcG\cup\{a\}$ which does not fixate $a$ will be extendable to an automorphism, by the $\geq$$k-$homogeneity of $\mcM$. Thus $a\in q$, which contradicts that $a\in p$.
\end{proof}
We will keep notation from the previous Lemma and let $p$ be the infinite $1-$orbit and let $q$ be the finite $1-$orbit in $\mcM$.
\begin{lma}\label{pqnonadj} If $a\in p$ and $b\in q$ then $a$ is not adjacent to $b$.
\end{lma} 
\begin{proof}
If some element in $p$ is adjacent to some element in $q$ then for each $a_0\in p$ there exists some $b_0\in q$ such that $a_0$ is adjacent to $b_0$. As $p$ is infinite and $q$ is finite, there has to exist $b\in q$ such that $b$ is adjacent to an infinite amount of vertices in $K_\infty\subseteq\mcM$. However, if $\mcG\subseteq K_\infty\subseteq\mcM$ with $|\mcG|=k$ such that all vertices in $\mcG$ are adjacent to $b$ then there is an embedding $f:\mcG\cup\{b\}\rightarrow \mcG\cup\{b\}$ which does not fixate $b$. Since $\mcM$ is $\geq$$k-$homogeneous it is possible to extend $f$ to an automorphism of $\mcM$. Thus $b\in p$ which is a contradiction.
\end{proof}
As Lemma \ref{pqnonadj} proves that $p$ and $q$ are not connected to each other, the next lemma shows that each element in $p$ is non-adjacent to at most $k-1$ elements in $p$.
\begin{lma} If $a\in p$ then there are at most $k+|q|-1$ elements in $\mcM$ which $a$ is not adjacent to.
\end{lma}
\begin{proof}
Assume $a\in p$ is not adjacent to any elements in some $\mcG' \subseteq \mcM$ such that $|\mcG'|=k+|q|$ and let $\mcG= \mcG'\cap p$. By Lemma \ref{pqnonadj} no element in $\mcG$ is adjacent to any element in $q$. Assume $b\in q$. The function $f:\mcG\cup\{a\}\rightarrow \mcG\cup\{b\}$ mapping $\mcG$ to $\mcG$ and $a$ to $b$ is thus an embedding. Since $|\mcG|\geq k$ it is possible, by $\geq$$k-$homogeneity, to extend $f$ into an automorphism. It follows that $a\in q$ which is a contradiction.
\end{proof}
$K_\infty(a)$ is any subgraph $\mcG$ of $\mcM$ which is isomorphic to $K_\infty$ and contains $a$. Thus the following lemma proves that $b$ is adjacent to all elements (except $a$) in all $K_\infty\subseteq \mcM$ which contain $a$.
\begin{lma}\label{pnonadjlma}
If $a,b\in p$ are such that $a$ is not adjacent to $b$ then $b$ is adjacent to each element in $K_\infty(a)-\{a\}$.
\end{lma}
\begin{proof}
Assume that each element in $p$ is non-adjacent to exactly $t$ other elements in $p$. Assume in search for a contradiction that $d\in K_\infty(a)-\{a\}$ is not adjacent to $b$. All pairs of distinct elements from $K_\infty(a)$ belong to the same $2-$orbit. Thus for any distinct $\alpha,\beta\in K_\infty(a)$ there exists $\gamma\in p$ such that $\alpha$ and $\beta$ are both not adjacent to $\gamma$. This however implies either that $a$ would be non-adjacent to more than $t$ different elements or that there would exist some element $c\in p$ which is non-adjacent to more than $t$ elements in $K_\infty(a)$. Both of these conclusions lead to a contradiction since we have assumed each element in $p$ to be non-adjacent to exactly $t$ other elements in $p$.
\end{proof}
\begin{lma}\label{notadjlma}
If $a,b,c\in p$ are such that $a$ is not adjacent to $b$ and $a$ is not adjacent $c$, then $b$ is not adjacent to $c$.
\end{lma}
\begin{proof}
Assume that $b$ is adjacent to $c$. Lemma \ref{pnonadjlma} implies that both $b$ and $c$ are adjacent to each element in $K_\infty(a)-\{a\}$. Thus $(K_\infty(a)\cup\{b,c\})-\{a\}$ is a complete graph, where both $b$ and $c$ are non-adjacent to $a$. As $b$ is not adjacent to $a$ and $c\in K_\infty(b,c)$, Lemma \ref{pnonadjlma} implies that $c$ is adjacent to $a$ which is a contradiction.
\end{proof}
\begin{lma}\label{adjlma}If $a,b\in p$ and $a$ is adjacent to $b$ then there exists $\mcG\subseteq \mcM$ such that $a,b\in \mcG$ and $\mcG\cong K_\infty$.
\end{lma}
\begin{proof}
If $a$ is adjacent to an infinite amount of elements in some subgraph $\mcG\subseteq \mcM$ such that $\mcG\cong K_\infty$ and $b\in \mcG$, then the lemma holds. Assume, in search for a contradiction, that there exist elements $c,d\in K_\infty(b)$ such that $a$ is not adjacent to both $c$ and $d$. Lemma \ref{notadjlma} then implies that $c$ is not adjacent to $d$, which is a contradiction.
\end{proof}
We now summarize all our knowledge about $p$ and $q$ into the following lemma which proves the second part of this section's main Lemma \ref{thm1}.
\begin{lma}\label{fin1lma}For some $t\in \mbbZ^+$, $p\cong \mcG_t$ and $\mcM\cong \mcG_t\dot\cup \mcH$ for some homogeneous finite graph $\mcH$.
\end{lma}
\begin{proof}Assume that each element $a\in p$ is non-adjacent to $t$ elements $b_1,\ldots, b_t$ $\in p$. By Lemma \ref{notadjlma} these $t$ elements always form a $K_t^c$. If $c\in p$ is adjacent to $b_i$, for some $i\in \{1,\ldots,t\}$, then Lemma \ref{pnonadjlma} and Lemma \ref{adjlma} implies that $c$ is adjacent to all of $b_1,\ldots,b_t$. It is thus clear that $p\cong \mcG_t$.\\\indent
Let $q'\subseteq q$, let $g : q'\to q$ be any embedding and choose $p'\subseteq p$ such that $|p'|=k$. The function $f:p'\cup q'\to \mcM$ which maps $p'$ by inclusion to $p'$ and $q'$ according to $g$ to $q$ is then an embedding since, by Lemma \ref{pqnonadj}, elements in $p$ and $q$ are not adjacent. As $\mcM$ is $\geq$$k-$homogeneous and $|p'\cup q'|\geq k$, $f$ is possible to extend into an automorphism $f'$ of $\mcM$. Now $f'$ maps $q$ to $q$, thus if we restrict $f'$ to $q$, we get an automorphism of $q$ which by definition extends $g$. Hence we have shown that $q$ satisfies the definition of being homogeneous and we can conclude that $\mcM\cong \mcG_t \dot\cup \mcH$ for some finite homogeneous graph $\mcH$.
\end{proof} 
Using the tools we have developed in this section, we can now finally prove the main lemma. Note that we do not use any other assumptions than those stated in the formulation of the lemma.
\begin{proof}[Proof Lemma \ref{thm1}]
Lemma \ref{fin1lma} proves that if $\mcM$ is $\geq$$k-$homogeneous and $K_\infty\subseteq \mcM$ then $\mcM \cong \mcG_t\dot\cup \mcH$ for some $t\in\mbbN$ and homogeneous finite graph $\mcH$. It follows that, using Ramsey's theorem, if $K_\infty\not\subseteq \mcM$ then $K_\infty^c\subseteq \mcM$ in which case it follows similarly that $\mcM\cong (\mcG_t\dot\cup \mcH)^c$.\\\indent 
In order to prove the second direction assume that $\mcM\cong \mcG_t\dot\cup \mcH$ and notice that both $\mcG_t$ and $\mcH$ are homogeneous graphs and as subgraphs of $\mcM$ they constitute distinct $1-$orbits which are not connected to each other. Let $\mcG'\subseteq \mcM$ be a finite subgraph such that $ |\mcG'| \geq k = 3\max(|\mcH|,t)$. 
There are at least $2t$ vertices in $\mcG'$ which belong to $\mcG_t$, thus there exists a vertex $a\in \mcG'$ which is adjacent to at least $|\mcH|$ vertices in $\mcM$. Furthermore if $b\in \mcG_t$ then $b$ is adjacent either to $a$ or an element in $\mcG'$ which $a$ is adjacent to. This implies that $\mcG'\cap \mcG_t$ has to consist of more than $|\mcH|$ vertices  which are all in a single connected component. Thus any embedding $f:\mcG'\rightarrow \mcM$ has to map $\mcG'\cap \mcG_t$ into $\mcG_t$ and $\mcG'\cap \mcH$ into $\mcH$. As $\mcH$ and $\mcG_t$ are both homogeneous this means that $f$ can be extended to an automorphism of $\mcM$, thus $\mcM$ is $\geq$$k-$homogeneous. 
\end{proof}

\section{Graphs which are $1-$homogeneous but not $2-$homogeneous}\label{sec2}
\begin{lma}\label{twothm}
A countably infinite graph $\mcM$ is $\geq$$k-$homogeneous, for some $k\in\mbbN$, $1-$homogeneous but not $2-$homogeneous if and only if there exists $t$ such that $\mcM \cong \mcH_{t,1}$, $\mcM\cong \mcH_{t,2}$, $\mcM\cong \mcH_{t,1}^c$ or $\mcM\cong\mcH_{t,2}^c$.
\end{lma} The proof is left for the end of this section.
In order to prove the second direction of Lemma \ref{twothm}, we will assume throughout the rest of this section that $\mcM$ is $\geq$$k-$homogeneous, $1-$homo\-ge\-ne\-ous but not $2-$homogeneous i.e. there are more than three $2-$orbits but only a single $1-$orbit. Due to Ramsey's theorem $\mcM$ has to contain either $K_\infty$ or $K_\infty^c$. We will assume that $\mcM$ contains $K_\infty$ and the reader may notice that all the reasoning in this section may be done in the same way for $K_\infty^c$ by switching all references to edges and non-edges. Since there is only a single $1-$orbit, writing $K_\infty(a)$ always makes sense for any vertex $a\in \mcM$, while writing $K_\infty(a,b)$ needs to be motivated in order to show existence.
\begin{lma}
There are at most two $2-$orbits containing tuples of adjacent elements in $\mcM$ and there are at most two $2-$orbits containing tuples of distinct non-adjacent elements in $\mcM.$
\end{lma}
\begin{proof}
Assume $(a,b_1)$, $(a,b_2)$ and $(a,b_3)$ are three different $2-$orbits such that $a$ is adjacent to $b_1,b_2$ and $b_3$. We may assume that $a$ is the first coordinate of all three parts without loss of generality, since we only have a single $1-$orbit in $\mcM$. One of the $2-$orbits may be assumed to be a part of a $K_\infty$, say $(a,b_1)$. Since $\mcM$ is $\geq$$k-$homogeneous this property is unique for the orbit of $(a,b_1)$. But then neither $b_2$ nor $b_3$ may be adjacent to more than $k-1$ of the elements in $K_{3k}(a,b_1)$. Thus we are able to find $\mcG\subseteq K_{3k}(a,b_1)$ such that $a\in \mcG$, $|G|\geq k$ and nothing in $\mcG$, except $a$, is adjacent to $b_2$ or $b_3$. The function $f:\mcG\cup\{b_1\}\rightarrow \mcG\cup\{b_3\}$ mapping $\mcG$ and $b_2$ to $\mcG$ and $b_3$ is an embedding, and hence the $\geq$$k-$homogeneity implies that $(a,b_2)$ and $(a,b_3)$ are of the same orbit, contradicting the assumption.
\\\indent
For the second part of the lemma, assume $(c,d_1)$, $(c,d_2)$ and $(c,d_3)$ are different $2-$orbits such that $c$ is non-adjacent to all of $d_1,d_2$ and $d_3$. Since the orbits are different the $\geq$$k-$homogeneity implies that $d_1$ is adjacent to at least $k$ of the vertices in $K_{2k}(c)$ if and only if $d_2$ is adjacent to at most $k-1$ vertices in $K_{2k}(c)$. However, $d_3$ has to be adjacent or non-adjacent to at least $k$ vertices in $K_{2k}(c)$ and hence the orbit of $(c,d_3)$ can't be distinct from the two other orbits by the $\geq$$k-$homogeneity of $\mcM$.
\end{proof}
The previous lemma implies that we may assume there are at most five $2-$orbits in $\mcM$, out of which one is the orbit containing identical element $2-$tuples $(x,x)$. Call the $2-$orbits where elements have an edge between them, $q_1$ and $q_2$ and assume that $q_1$ is the orbit of pairs of elements in $K_\infty$. It follows that $(a,b)\in q_2$ if $a$ is adjacent to less than $k-1$ elements in $K_\infty(b)$ and $a$ is adjacent to $b$.
\begin{lma}\label{qfinlma} For each $a\in \mcM$, there are only finitely many (possibly zero) elements $b\in \mcM$ such that $(a,b)\in q_2$.
\end{lma}
\begin{proof}
Let $\mcA_a\subseteq \mcM$ be the subgraph containing all elements $b$ such that $(a,b)\in q_2$ and assume in search for a contradiction that $\mcA_a$ is infinite. 
By Ramsey's theorem either $K_k\subseteq \mcA_a$ or $K_k^c\subseteq\mcA_a$. If $K_k\subseteq \mcA_a$ then, since $a$ is adjacent to each element in $\mcA_a$, for any $b\in K_k\subseteq \mcA_a$, $(a,b)\in q_1$ which is a contradiction against that $(a,b)\in q_2$.
\\\indent On the other hand assume that $K_k^c\subseteq \mcA_a$ and $b\in K_k^c$. Each element in $\mcA_a$ is adjacent to less than $k-1$ elements in $K_\infty(a)$, thus there exists a vertex $c\in K_\infty(a)$ such that none of the elements in $K_k^c\subseteq \mcA_a$ is adjacent to $c$. The function $f:K_k^c\cup\{a\}\rightarrow (K_k^c-\{b\})\cup\{a,c\}$ mapping $(K_k^c-\{b\})\cup\{a\}$ back to itself pointwise and $b$ to $c$ is then an embedding. Thus $\geq$$k-$homogeneity implies that $(a,b)\in q_1$, which is a contradiction.
\end{proof}
Call the $2-$orbits of tuples of distinct elements which have no edge between them $p_1$ and $p_2$. Assume $p_1$ is the orbit of pairs $(a,b)$ such that $b$ is adjacent to at most $k-1$ of the elements in $K_\infty(a)$. We note that $p_1$ has to exist, since else $q_2$ can't exist, which would imply that $\mcM$ has at most three $2-$orbits. It follows, using the $\geq$$k-$homogeneity, that each element $b$ which is non-adjacent to at least $k-1$ elements in $K_\infty(a)$ is such that $(a,b)\in p_1$. Thus the orbit $p_2$ contains all pairs $(a,b)$ such that $a$ and $b$ are non-adjacent yet $b$ is non-adjacent to less than $k-1$ elements in $K_\infty(a)$. It follows quickly from the definition, and the $\geq$$k-$homogeneity, that the four orbits $p_1,p_2,q_1$ and $q_2$ are symmetric in the sense that $(a,b)\in r$ implies $(b,a)\in r$.
\begin{lma}\label{lmaE} Let $a,b,c\in \mcM$. If $(a,b)\in q_1, (a,c)\in p_1$ and $b$ is not adjacent to $c$ then $(b,c)\in p_1$.
\end{lma}
\begin{proof}
Since $(a,c)\in p_1$, $c$ is adjacent to at most $k-1$ elements in $K_\infty(a,b)$, thus $(b,c)\in p_1$.
\end{proof}
We are now ready to prove that, similarly to $q_2$, the orbit $p_2$ is finite if we fix one component.
\begin{lma}\label{lma5}
For each $a\in \mcM$ there are only finitely many (possibly zero) elements $b\in \mcM$ such that $(a,b)\in p_2$.
\end{lma}
\begin{proof}
Assume $c$ is such that $(a,c)\in p_1$, let $L$ be the set of all elements $b\in \mcM$ such that $(a,b)\in p_2$ and assume that $L$ is infinite. By Ramsey's theorem, we either have $K_\infty\subseteq L$ or $K_\infty^c\subseteq L$. \\\indent If $K_\infty\subseteq L$ then all these elements are non-adjacent to $a$ but then the definition of $p_1$ implies that $(a,b')\in p_1$ for each $b'\in K_\infty\subseteq L$. This is a contradiction against $(a,b')\in p_2$. 
\\\indent
Assume instead that $K_{\infty}^c\subseteq L$ and let $\mcG\subseteq K_\infty^c\subseteq L$ be such that all elements in $\mcG$ are non-adjacent to $c$. If $|\mcG|\geq k-2$ then any injective function $f:\mcG\cup\{a,c\}\rightarrow \mcG\cup\{a,c\}$ is an embedding, thus $\geq$$k-$homogeneity implies that $(a,c)$ is in the same orbit as $(a,d)$ for any $d\in \mcG$. This is a contradiction, since $(a,d)\in p_2$ and $(a,c)\in p_1$, thus $|\mcG|\leq k-3$. 
We can hence find $\mcH\subseteq K_\infty^c\subseteq L$ such that $|\mcH|=2k$ and $c$ is adjacent to all elements in $\mcH$. All elements in $\mcH\subseteq L$ are, by the definition of $L$, non-adjacent to at most $k-2$ elements in $K_\infty(a)$, thus there exists an element $e\in K_\infty(a)$ such that $e$ is adjacent to all elements in $\mcH$.
Assume without loss of generality that $b\in \mcH$. There are embeddings $g:\mcH\cup\{c\}\rightarrow \mcH\cup\{a,e\}$ which map $(b,c)$ to $(a,e)$. This, together with $\geq$$k-$homogeneity, implies that $(b,c)\in q_1$. Lemma \ref{lmaE} together with $(b,c)\in q_1$, $(a,c)\in p_1$ and $(a,b)\in p_2$ implies that $(a,b)\in p_1$ which is a contradiction.
\end{proof}

The next lemma shows that the orbits $q_1$ and $p_2$, in some sense, are closed and together form a tight part of the graph $\mcM$. This is a vital property which will be used many times in order to handle $q_1$ and $p_2$ in the rest of the section. 
\begin{lma}\label{lma6a} Let $a,b,c\in \mcM$. If $(a,b)\in q_1$ and $(a,c)\in p_2$ then $(b,c)\in q_1$.
\end{lma}
\begin{proof}
If $b$ is not adjacent to $c$ then, for every $a'\in K_\infty(a)$, since $(a',a)\in q_1$, there has to exist an element $c'$ such that $(a,c'), (a',c')\in p_2$. But this contradicts Lemma \ref{lma5}, since each element $c_0$ such $(a,c_0)\in p_2$ is non-adjacent to at most $k-1$ elements in $K_\infty(a)$. Thus we conclude that only $a$ in $K_{\infty}(a)$ can be non-adjacent to $c$, hence $b$ is adjacent to $c$ and more specifically $(b,c)\in q_1$.\\\indent
\end{proof}

\begin{lma}\label{lma6b} Let $a,c,d\in \mcM$. If $c\neq d$ and $(a,c), (a,d)\in p_2$ then $(c,d)\in p_2$.
\end{lma}
\begin{proof}
Assume $(c,d)\notin p_2$ and note that $(a,c),(a,d)\in p_2$ implies that $c$ and $d$ are non-adjacent to a finite amount of elements in $K_\infty(a)$. Thus $c$ is adjacent to an infinite amount of elements in $K_\infty(d)$ and hence the only orbit which $(c,d)$ can be a part of, out of $p_1,q_1$ and $q_2$, is $q_1$. By Lemma \ref{lma6a} it follows that $(c,d)\in q_1$ and $(c,a)\in p_2$ implies $(d,a)\in q_1$, which is a contradiction against $(a,d)\in p_2$.
\end{proof}
It is much harder to get a grip of the orbits $q_2$ and $p_1$. This is a consequence of that we assumed $K_\infty\subseteq \mcM$ and thus having neighbors which are also adjacent to some element is easy to handle. The rest of the section will be dedicated to reasoning out how these orbits work in $\mcM$.
\begin{lma}\label{lma7}
For each $a,b \in \mcM$ if $(a,b)\in q_2$ then each element $c\in K_{\infty}(a)-\{a\}$ will be such that $(b,c)\in p_1$. 
\end{lma}
\begin{proof}
It is clear that $(b,c)\notin q_1$, since we otherwise would have a contradiction against the facts that $(a,c)\in q_1$ and $(a,b)\in q_2$. Assume in search for a contradiction that $(b,c)\in q_2$. For every element $c_0\in K_{\infty}(a)$, $(c_0,a)$ is in the same $2-$orbit as $(c,a)$ thus there exists an element $d_0$ which is to $(c_0,a)$ as $b$ is to $(c,a)$, thus we know that $(a,d_0),(c_0,d_0)\in q_2$. This implies either that there is an infinite amount of elements $d'$ such that $(a,d')\in q_2$ or that there is an element $d''$ such that $(d'',c_0)\in q_2$ for an infinite amount of elements $c_0\in K_\infty(a)$. However both of these conclusions are contradictions against Lemma \ref{qfinlma}. Thus $c$ is not adjacent to $b$. By the definition of $q_2$, there exists some $\mcG\subseteq K_\infty(a)$ such that $|G|= k$ and each element in $\mcG$ is non-adjacent to $b$, thus $(b,c)\in p_1$.
\end{proof}
In the upcoming two lemmas we will show that the orbit $q_2$ and the orbit $p_2$ are very closely linked, and in fact most cases where $q_2$ exist, also $p_2$ has to exist.
\begin{lma}\label{lma8}
If $a,b_1,b_2\in \mcM$, $b_1\neq b_2$ and $(a,b_1),(a,b_2)\in q_2$ then $(b_1,b_2)\in p_2$.
\end{lma}
\begin{proof}
By Lemma \ref{lma7} if $d\in K_{\infty}(b_1)$ then $d$ can at most be adjacent to $k-1$ elements in $K_{\infty}(a)$. If $(b_1,b_2)\in q_1$ then $b_1\in K_\infty(b_2)$ which together with Lemma \ref{lma7} implies that $(a,b_1)\in p_1$ which is a contradiction.\\\indent Assume instead that $(b_1,b_2)\in q_2$. Choose $\mcG\subseteq K_\infty(a)$ and $d\in K_\infty(b_1)$ such that $|G|=k$ and all elements in $\mcG$ are non-adjacent to $b_1,b_2$ and $d$. The function $f:\mcG\cup\{b_1,d\}\rightarrow \mcG\cup\{b_1,b_2\}$ mapping $\mcG$ to $\mcG$ and $(b_1,d)$ to $(b_1,b_2)$ is then an embedding, thus $\geq$$k-$homogeneity implies that $(b_1,b_2)$ and $(b_1,d)$ belong to the same orbit which is a contradiction as $(b_1,d)\in q_1$. We conclude that $b_1$ must be nonadjacent to $b_2$.
\\\indent
Assume $(b_1,b_2)\in p_1$ and let $c\in K_\infty(a)$. It is then possible to find $\mcG\subseteq K_{\infty}(b_1)$ such that $|G|=k$ and all elements in $\mcG$ are non-adjacent to $a,c$ and $b_2$. If $f:\mcG\cup\{a,c\}\rightarrow \mcG\cup\{a,b_2\}$ maps $\mcG$ to $\mcG$ and $(a,c)$ to $(a,b_2)$ then the $\geq$$k-$homogeneity implies that $(a,c)\in q_2 $, which is a contradiction.
\end{proof}
\begin{lma}\label{lma10}
Let $a,b,c\in \mcM$. If $(a,b)\in q_2$ and $(b,c)\in p_2$ then $(a,c)\in q_2$.
\end{lma}
\begin{proof}
Using Lemma \ref{lma6a} and $(b,c)\in p_2$ we get a contradiction against $(a,b)\in q_2$ if $(a,c)\in q_1$. If $(a,c)\in p_2$ we get a contradiction against $(a,b)\in q_2$ using Lemma \ref{lma6b} and $(b,c)\in p_2$.
Thus we know that $(a,c)\notin q_1$ and $(a,c)\notin p_2$.\\\indent Assume, in search for a contradiction, that $(a,c)\in p_1$. Let $d \in K_{\infty}(b)-\{b\}$ and note by Lemma \ref{lma7} that $(a,d)\in p_1$. Since $(a,d)$ and $(a,c)$ are in the same orbit, there has to exist an element $e\in \mcM$, corresponding to what $b$ is to $(a,d)$, such that $(e,c)\in q_1$ and $ (e,a)\in q_2$. Lemma \ref{lma8} now implies that $(e,b)\in p_2$ which in turn together with Lemma \ref{lma6b} implies that $(e,c)\in p_2$. But $(e,c)\in q_1$, hence this is a contradiction and we can conclude that $(a,c)\in q_2$.
\end{proof}
Lastly we figure out how the orbit $p_1$ behaves. This is the hardest orbit to handle, as it induces so little information about edges. 
\begin{lma}\label{lma11}
Let $a,b,c\in \mcM$. If $(a,b)\in p_1$ and $(a,c)\in p_2$ then $(b,c)\in p_1$.
\end{lma}
\begin{proof}
Lemma \ref{lma6a} implies that $(b,c)\notin q_1$, Lemma \ref{lma6b} implies that $(b,c)\notin p_2$ and Lemma \ref{lma10} implies that $(b,c)\notin q_2$, thus we conclude that $(b,c)\in p_1$.
\end{proof}
\begin{lma}\label{lma9}
Let $a,b,c\in \mcM$ with $b\neq c$. If $(a,b),(a,c)\in p_1$ then $(b,c)\in q_1$ or $(b,c)\in p_2$.
\end{lma}
\begin{proof}
Assume $(b,c)\in q_2$ and let $\mcG\subseteq K_\infty(a)$ be such that $|G|=k$ and both $b$ and $c$ are non-adjacent to each element in $\mcG$. By the definition of $p_1$ there exists $d\in K_\infty(b)$ such that $d$ is non-adjacent to each element in $\mcG$ and by Lemma \ref{lma7} we know that $d$ is not adjacent to c. The function $f:\mcG\cup\{b,d\}\rightarrow \mcG\cup\{b,c\}$ mapping $\mcG\cup\{b\}$ pointwise to $\mcG\cup\{b\}$ and $d$ to $c$ is then an embedding, which by the $\geq$$k-$homogeneity may be extended into an automorphism. This is a contradiction since $(b,c)\in q_2$ but $(b,d)\in q_1$.\\\indent
Assume for the rest of this proof that $(b,c)\in p_1$. In order to reach a contradiction in this case we will also need to make assumptions on which of the orbits $q_2$ and $p_2$ exists. Assume that $p_2$ exist, let $d$ be such that $(a,d)\in p_2$. Lemma \ref{lma11} implies that $(b,d)\in p_1$, thus d is adjacent to at most $k-1$ elements in $K_\infty(b)$. We may then find $\mcG\subseteq K_\infty(b)$ such that $|G|=k$ and all elements in $\mcG$ are non-adjacent to $a,c$ and $d$. Thus the function $f:\mcG\cup\{a,d\}\rightarrow \mcG\cup \{a,c\}$ mapping $\mcG\cup\{a\}$ to itself pointwise and $d$ to $c$ is an embedding. Since $\mcM$ is $\geq$$k-$homogeneous $f$ may be extended into an automorphism which implies that $(a,d)\in p_1$ which is a contradiction.\\\indent
Assume $q_2$ exists and that $d$ is such that $(a,d)\in q_2$. If both $b$ and $c$ are non-adjacent to less than $k$ elements in $K_\infty(d)$ then there is an infinite $\mcG\subseteq K_\infty(d)$ such that both $b$ and $c$ are adjacent to all elements in $\mcG$. This however contradicts that $(b,c)\in p_1$, thus at least one of $b$ or $c$ is non-adjacent to more than $k$ elements in $K_\infty(d)$ and thus $b$ or $c$ is adjacent to at most $k-1$ elements in $K_\infty(d)$. From Lemma \ref{lma8} it follows that $(d,b),(d,c)\notin q_2$. Thus at least one of $(b,d)$ and $(c,d)$ belong to $p_1$. Assume $(c,d)\in p_1$ (the case $(b,d)\in p_1$ is similar). 
We may then find $\mcG\subseteq K_\infty(a)$, $\mcH\subseteq K_\infty(c)$ and $e\in K_\infty(c)$ such that $e$ is not adjacent to $d$, $|G|=|H|=k$, each vertex in $\mcG$ is not adjacent to $c$ or $e$ and each vertex in $\mcH$ is not adjacent to $d$ or $a$. 
The function $f:\mcH\cup\{a,d\}\rightarrow \mcG\cup\{c,e\}$ mapping $\mcH$ to $\mcG$, $a$ to $e$ and $d$ to $c$ is then an embedding. Thus $\geq$$k-$homogeneity implies that $f$ may be extended into an automorphism and thus $(a,d)\in q_1$ which is a contradiction.
\end{proof}
We will now put together our previous knowledge in to a lemma which gives us the second part in proving Lemma \ref{twothm}. Recall the definition of $\mcG_t$ from the introduction. 
\begin{lma}\label{mainlma}
The following hold for $\mcM$:
\begin{itemize}
\item If $p_2$ does not exist then $\mcM\cong \mcH_{1,2}$.
\item If $q_2$ does not exist then for some $n\geq 2$, $\mcM\cong \mcH_{n,1}$.
\item If both $p_2$ and $q_2$ exist then for some $n\geq 2$, $\mcM\cong \mcH_{n,2}$
\end{itemize}
\end{lma}
\begin{proof}
Assume $p_2$ does not exist. Lemma \ref{lma8} implies that for each $a\in \mcM$ there is a unique element $b$ such that $(a,b)\in q_2$. For each element there is a $K_\infty$ containing it, and by what we know about $p_1$ there are at least two disjoint such. Lemma \ref{lma9} implies that if $a$ is non-adjacent to both $b$ and $c$ then $b,c$ are both contained in the same $K_{\infty}$. Thus we conclude that there are exactly two disjoint copies of $K_\infty$ such that each vertex is connected to exactly one vertex in the other $K_\infty$. This implies that $\mcM\cong \mcH_{1,2}$.
\\\indent 
Assume $q_2$ does not exist, and assume that for each $a$ there are exactly $n-1$ elements $b$ such that $(a,b)\in p_2$. Each element is contained in at least one $K_\infty$, Lemma \ref{lma6a} implies that there are exactly $n-1$ elements which $a$ is non-adjacent to and for which $a$ is exchange\-able in $K_\infty(a)$ i.e. we have found a $\mcG_n$ subgraph. Since $p_1$ exists there has to be at least two copies of $\mcG_n$ in $\mcM$, however Lemma \ref{lma9} implies that there are exactly two, hence $\mcM\cong\mcH_{n,1}$.
\\\indent
Assume both $p_2$ and $q_2$ exist. By the same reasoning as in the previous case when $p_2$ exists, we get that there exist two copies of $\mcG_n$ for some $n\geq 2$. We however also have the existence of $q_2$, and by Lemma \ref{lma8} and Lemma \ref{lma10} we know that for each distinct $a,b,c$ such that $(a,b)\in q_2$, $(a,c)\in q_2$ if and only if $(b,c)\in p_2$. This implies that $\mcM\cong \mcH_{n,2}$.
\end{proof}
Now we can finally prove the main lemma of this section. Note that we in this proof do not assume anything except what is stated in the formulation of the lemma.
\begin{proof}{Proof Lemma \ref{twothm}}
We prove that the suggested graphs are actually $\geq$$k-$ho\-mogeneous, for some $k$, and the other direction is done in Lemma \ref{mainlma}. We first note that if $\mcM = \mcH_{t,1}$ or $\mcM = \mcH_{t,2}$ then there are two parts in $\mcM$, each of which is isomorphic to $\mcG_r$ for some $r$. If $\mcM \cong \mcH_{t,1}$ and $\mcG\subseteq \mcM$ such that $\omega >|G|\geq 2t+1$ there has to be at least one edge between some elements $a,b\in G$. Thus if $f:\mcG\rightarrow \mcM$ is an embedding, each vertex adjacent to $a$ or $b$ will then be mapped to one of the $\mcG_r$ parts of $\mcM$ and each vertex not adjacent $a$ nor $b$, will be mapped the other $\mcG_r$ part. As no edges exist between the two parts and each part is a homogeneous graph $f$ may be extended into an automorphism. \\\indent
If $\mcM\cong \mcH_{t,2}$ and $\mcG\subseteq \mcM$ with $\omega>|G|\geq 4t+1$ then there exist vertices $a,b,c\in \mcG$ which are adjacent to each other. If $f:\mcG\rightarrow \mcM$ then $a,b$ and $c$ needs to be mapped to the same part. An element $d\in \mcG$ is mapped to the same part as $a,b$ and $c$ if and only if it is adjacent to two of these vertices. As edges between parts are preserved by $f$ and each part is homogeneous, $f$ may be extended into an automorphism. \\\indent 
If $\mcM\cong \mcH_{t,1}^c$ or $\mcM\cong\mcH_{t,2}^c$ then the reasoning is equivalent.
\end{proof}
\section{$1-$ and $2-$homogeneous graphs}\label{sec3}
In this section we want to prove the following lemma which will finish the classification in Theorem \ref{thm}.
\begin{lma}\label{thirdmainlemma} Let $k\in\mbbN$. Each infinite graph $\mcM$ which is $\geq$$k-$homogeneous, $1-$homogeneous and $2-$homogeneous is homogeneous.
\end{lma}
In order to prove this lemma, we assume that $\mcM$ is an infinite $\geq$$k-$homo\-geneous graph which is $1-$homogeneous and $2-$homogeneous such that there are finite $\mcG_1,\mcG_2\subseteq \mcM$ such that $\mcG_1\cong\mcG_2$ and yet $\mcG_1$ and $\mcG_2$ are not of the same orbit in $\mcM$. Let $n=|\mcG_1|$ and assume that $\mcM$ is $<$$n-$homogeneous, thus $\mcG_1$ is one of the smallest subgraphs of $\mcM$ whose isomorphism type does not determine its orbit. Due to Ramsey's theorem either $K_\infty$ or $K_\infty^c$ is embeddable in $\mcM$. We will assume that $K_\infty$ is embeddable in $\mcM$, and the reader may notice that all arguments can be carried out in the same way by changing all references to edges by non-edges and vice versa, in the case where $K_\infty^c$ is embeddable instead. 
\\\indent
Let $a\in\mcG_1$ and put $\mcG=\mcG_1-\{a\}$. As $\mcM$ is $(n-1)-$homogeneous we know that if $f:\mcG_1\rightarrow\mcG_2$ is an isomorphism, then the orbit of $\mcG$ in $\mcM$ is the same as the orbit of $\mcG_2-f(a)$. Thus we conclude that there is an element $b\in \mcM$ such that $\mcG\cup\{b\}$ is in the same orbit as $\mcG_2$, when $b$ is mapped to $f(a)$.
\begin{lma}\label{notincommonSpecial} There is no $\mcH\subseteq \mcM$ and $\mcH\cong K_{(k+n)}$ such that both $a$ and $b$ are adjacent (or non-adjacent) to all elements in $\mcH$. 
\end{lma}
\begin{proof} 
If there exists such a graph $\mcH$ then let $\mcH_0\subseteq \mcH$ such that $|\mcH_0|=k-1$, $\mcH_0\cap (\mcG\cup\{a,b\})=\emptyset$ and let $f:\mcH_0\cup\mcG\cup\{a\}\to \mcH_0\cup\mcG\cup\{b\}$ map $\mcH_0\cup\mcG$ pointwise to itself and map $a$ to $b$. The function $f$ is clearly an isomorphism and thus the $\geq$$k-$homogeneity implies that $\mcG\cup\{a\}$ and $\mcG\cup\{b\}$ are in the same orbit, which is a contradiction.
\end{proof}
It is clear from the proof that the previous lemma also works if we replace $(k+n)$ in $K_{(k+n)}$ with some larger number or infinity. This will be used later.
\begin{cor} $a$ is not adjacent to $b$.
\end{cor}
\begin{proof}
The $2-$orbit of elements which are adjacent to each other is uniquely determined by its isomorphism type, and as such an orbit exists in $K_\infty$ it follows that if $a$ and $b$ were adjacent to each other then there would be a graph $\mcH\subseteq\mcM$ such that $\mcH\cong K_\infty$ and both $a$ and $b$ are adjacent to all elements in $\mcH$, contradicting Lemma \ref{notincommonSpecial}.
\end{proof}
The previous corollary together with the fact that we only have a single $2-$orbit for distinct non-adjacent elements implies the following generalization of Lemma \ref{notincommonSpecial}.
\begin{lma}\label{notincommon} 
Let $\alpha,\beta\in M$ such that $\alpha\neq \beta$ and $\alpha $ is not adjacent to $\beta$. There is no $\mcH\subseteq \mcM$ and $\mcH\cong K_{(k+n)}$ such that both $\alpha$ and $\beta$ are adjacent (or non-adjacent) to all elements in $\mcH$. 
\end{lma}

\begin{lma}$a$ and $b$ are adjacent to all elements in $\mcG$.
\end{lma}
\begin{proof}
Assume $a$ is not adjacent to some element $c\in\mcG$. As $\mcG_1\cong\mcG_2$ it follows that $b$ is non-adjacent to the same element. Let $\mcH\subseteq \mcM$ be such that $a$ is adjacent to all elements in $\mcH$ and $\mcH\cong K_\infty$. By Lemma \ref{notincommon} we may assume that $b$ and $c$ are not adjacent to any elements in $\mcH$. However again using Lemma \ref{notincommon} but now on $(b,c)$ give us a contradiction.
\end{proof}

\begin{cor} $\mcG_1\cong K_n$.
\end{cor}
\begin{proof}The element $a$ is just an arbitrary element chosen in $\mcG_1$ and $a$ is adjacent to all other elements in $\mcG_1$, thus the result follows.
\end{proof}
\noindent We know that one of the orbits for tuples whose isomorphism type is $K_n$ is included in some $K_\infty$, thus assume that $\mcG_1$ is such. It follows that there is $\mcH_b\subseteq \mcM$ such that $\mcH_b\cong K_\infty$, $a\in \mcH_b$ and $\mcG\subseteq \mcH_b$. By Lemma \ref{notincommon} $b$ may not be adjacent to more than $k+n-1$ elements in $\mcH_b$ where the elements of $\mcG$ are included. We may however assume that the only elements in $\mcH_b$ which $b$ is adjacent to are the elements in $\mcG$ by choosing to not include all elements which $b$ are adjacent to in $\mcH_b$. \\\indent
Let $c,d$ be some distinct elements in $\mcG$. As $(\mcG-\{c\})\cup\{b\}$ is in the same $(n-1)$-orbit as $\mcG$, there exists some subgraph $\mcH_c\subseteq \mcM$ such that $\mcH_c\cong K_\infty $ and $(\mcG-\{c\})\cup\{b\}\subseteq \mcH_c$. If $c$ would be a part of $\mcH_c$ or adjacent to $k$ or more elements in $\mcH_c$ we would be able to create an embedding mapping $\mcG\cup \{b\}$ to $\mcG\cup\{a\}$ mapping at least $k$ elements, thus the $\geq$$ k-$homogeneity gives a contradiction. Thus $c$ is adjacent to at most $k$ elements in $\mcH_c$. For any element $\gamma\in \mcH_b-\mcG$ we know that $b$ is not adjacent to $\gamma$, thus Lemma \ref{notincommon} implies that $\gamma$ is adjacent to at most $k+n-1$ elements in $\mcH_c$. In the same way we may find a graph $\mcH_d\subseteq \mcM$ such that $\mcH_d\cong K_\infty$, $(\mcG-\{d\})\cup\{b\}\subseteq \mcH_d$ and $d$ is adjacent to at most $k$ elements in $\mcH_d$. If $e\in \mcH_d$ and $e$ is not adjacent to $d$ then it follows from Lemma \ref{notincommon} that $e$ is adjacent to at most $k+n-1$ elements in both $\mcH_b$ and $\mcH_c$ as $d$ is adjacent to all elements in both of those graphs. However if we let $e'\in \mcH_c$ be such that $e'$ is not adjacent to $e$ and $e'$ is adjacent to at most $k$ elements in $\mcH_b$ then both of $e$ and $e'$ are non-adjacent to some $K_\infty\subseteq \mcH_b$. But this contradicts Lemma \ref{notincommon}, thus the proof of Lemma \ref{thirdmainlemma} is complete.

\end{document}